\documentclass[english,a4paper]{amsart}
\pdfoutput=1

\usepackage[T1]{fontenc}
\usepackage{ae}
\usepackage{aecompl}
\usepackage{amssymb}

\input{xy} 
\xyoption{all}
\SelectTips{cm}{}
\usepackage{babel}
\usepackage[pdftex]{graphicx}    
\usepackage[leqno]{amsmath}
\usepackage{amsthm}

\renewcommand\phi{\varphi}
\renewcommand\epsilon{\varepsilon}
\renewcommand\theta{\vartheta}

\newcommand\mbb{\mathbb}

\newcommand\mcal{\mathcal}

\newcommand\ol{\overline}

\newcommand\ul{\underline}
\newcommand\wh{\widehat}
\newcommand\wt{\widetilde}

\newcommand\sH{\mcal{H}}
\newcommand\sI{\mcal{I}}

\newcommand\sO{\mcal{O}}
\newcommand\sP{\mcal{P}}

\newcommand\sS{\mcal{S}}

\newcommand\sV{\mcal{V}}

\newcommand\fm{\mathfrak{m}}

\newcommand\A{\mbb{A}}

\newcommand\N{\mbb{N}}
\renewcommand\P{\mbb{P}}

\newcommand\R{\mbb{R}}
\newcommand\Z{\mbb{Z}}

\renewcommand\div{{\rm div}}

\newcommand\ratmap{\dashrightarrow}

\DeclareMathOperator\Spec{Spec}

\DeclareMathOperator\PO{\rm PO}

\DeclareMathOperator\Sper{\rm Sper}

\newcommand\injects{\hookrightarrow}

\newcommand\into{\rightarrow}

\newcommand\iso{\xrightarrow{\sim}}
\newcommand\isom{\cong}

\renewcommand\le{\leqslant}
\renewcommand\ge{\geqslant}

\numberwithin{equation}{section}

\theoremstyle{plain}
\newtheorem{Thm}[equation]{Theorem}
\newtheorem{Prop}[equation]{Proposition}
\newtheorem{Cor}[equation]{Corollary}
\newtheorem{Lemma}[equation]{Lemma}

\newtheorem*{Thm*}{Theorem}
\newtheorem*{Satz*}{Satz}
\newtheorem*{Prop*}{Proposition}
\newtheorem*{Cor*}{Corollary}
\newtheorem*{Lemma*}{Lemma}
\newtheorem*{Hilfssatz*}{Lemma}
\newtheorem*{Sublemma*}{Sublemma}
\newtheorem*{Conjecture*}{Conjecture}

\theoremstyle{definition}
\newtheorem{Def}[equation]{Definition}

\newtheorem{Example}[equation]{Example}
\newtheorem{Examples}[equation]{Examples}

\newtheorem{Remark}[equation]{Remark}

\newtheorem*{Def*}{Definition}
\newtheorem*{Defs*}{Definitions}
\newtheorem*{Example*}{Example}
\newtheorem*{Examples*}{Examples}
\newtheorem*{LemmaDef*}{Lemma and Definition}
\newtheorem*{Notation*}{Notation}
\newtheorem*{Problem*}{Problem}
\newtheorem*{Question*}{Question}
\newtheorem*{Remark*}{Remark}
\newtheorem*{Remarks*}{Remarks}
\newtheorem*{Warning*}{Warning}

\newtheorem{Text}[equation]{}
\newtheorem*{Text*}{}

\renewcommand\nabla\triangledown

\begin{document}

\title{Sums of squares on reducible real curves}
\author{Daniel Plaumann}
\address{Fachbereich Mathematik\\ Universit{\"a}t Konstanz\\ 78457 Konstanz\\ Germany}
\email{Daniel.Plaumann@uni-konstanz.de}
\date{\today}
\begin{abstract}
We ask whether every polynomial function that is non-negative on a
real algebraic curve can be expressed as a sum of squares in the
coordinate ring. Scheiderer has classified all irreducible curves for
which this is the case. For reducible curves, we show how the answer
depends on the configuration of the irreducible components and give
complete necessary and sufficient conditions. We also prove partial
results in the more general case of finitely generated preorderings
and discuss applications to the moment problem for semialgebraic sets.
\end{abstract}

\maketitle
\section*{Introduction}

Let $V$ be an affine variety over $\R$, and let $V(\R)$ be its set of
real points. Can every polynomial function that is non-negative
(\emph{psd}) on $V(\R)$ be expressed as a sum of squares (\emph{sos})
in the coordinate ring $\R[V]$? If this is the case, we say that
\emph{psd$=$sos} in $\R[V]$. It has been known since Hilbert (1888)
that not every real psd polynomial in $n$ variables can be expressed
as a sum of squares of real polynomials, unless $n=1$. On the one
hand, Hilbert's negative result can be generalised to show that
psd$\neq $sos as soon as the dimension of $V(\R)$ is at least $3$; see
Scheiderer \cite{MR1675230}. This leaves the question of psd vs.~sos
for varieties of dimension at most $2$ where the answer depends on the
geometry of the curve or surface in question. On the other hand,
weaker statements can be proved in any dimension: For example, a
famous theorem of Schm\"udgen implies that if $V(\R)$ is compact, then
every strictly positive element of $\R[V]$ (i.e.~every psd function
without real zeros) is a sum of squares; see Schm\"udgen
\cite{MR1092173} or Prestel and Delzell \cite{MR1829790},
Thm.~5.1.17. This and similar statements of an approximative nature
have applications to polynomial optimization and functional
analysis. (General references on this topic are the books of Marshall
\cite{MR2383959}, Prestel and Delzell \cite{MR1829790}, and
Scheiderer's survey \cite{ScheidererSOSGuide}.)

The case when $V$ is one-dimensional (a real algebraic curve) is
completely understood if $V$ is irreducible, i.e.~if $V$ cannot be
expressed as a union of two non-empty curves. Scheiderer has shown
that psd$=$sos in the coordinate ring $\R[C]$ of a non-singular
irreducible real affine curve $C$ if and only if $C$ is rational or if
$C$ admits a non-constant bounded function (a function $f\in\R[C]$,
$f\notin\R$, such that $|f|\le n$ on $C(\R)$ for some $n\in\N$). He
calls the latter property \emph{virtual compactness} since the problem
of psd vs.~sos for such curves always reduces to the compact
situation. The singular case, too, is completely understood. For all
these results, see
Scheiderer \cite{MR1675230} and \cite{MR2020709}. It is the goal of this paper to extend them to the reducible case. 
Motivation, other than a desire for completeness, comes from Schm\"udgen's fibration theorem, which can be combined with our results to study the one- and two-dimensional moment problem of functional analysis.

\medskip

For the purpose of this introduction, let us consider the case of plane curves, which is conceptually simpler. Let $F\in\R[x,y]$ be a square-free polynomial with real coefficients, and let
\[
C(\R)=\bigl\{P\in\R^2\:|\: F(P)=0\bigr\}
\]
be the set of real points of the (affine) plane curve $C$ determined
by $F$. We will use the informal notation $C=\{F=0\}$. If $F=F_1\cdots
F_r$ is the factorization of $F$ in $\R[x,y]$ into its irreducible
factors, the curves $C_i=\{F_i=0\}$ yield the decomposition
$C=C_1\cup\cdots\cup C_r$ of $C$ into its ($\R$-)irreducible
components.  The coordinate ring $\R[C]$ is just the residue class ring $\R[x,y]/(F)$. Write $F=\sum F^{(j)}$ with $F^{(j)}$ homogeneous of degree $j$. If the point $(0,0)$ lies on $C$, recall that $(0,0)$ is called an \emph{ordinary double point} of $C$ if either $F^{(1)}\neq 0$ or $F^{(1)}=0$ and $F^{(2)}$ is a product of two distinct linear factors (up to scalar multiples) in $\R[x,y]$. Geometrically, this means that $(0,0)$ is either a non-singular point of $C$ or a singular point contained in exactly two branches with linearly independent tangents at $(0,0)$. With a change of coordinates, this notion extends to any real point of the plane. 

If $C$ is the union of two irreducible components $C_1$ and $C_2$ that intersect at a real ordinary double point $P$ and have no further intersection points, it is easy to describe polynomial functions on $C$ in terms of polynomial functions on $C_1$ and $C_2$, namely
\[
\R[C]\isom\biggl\{(f,g)\in\R[C_1]\times\R[C_2] \;\biggl|\; f(P)=g(P)\biggr\}.
\]
Now if $(f,g)$ is non-negative on $C(\R)$ and $f=\sum^n f_i^2$ and
$g=\sum^n g_i^2$ are sums of squares in $\R[C_1]$ resp.~$\R[C_2]$,
then $(f,g)=\sum^n (f_i,g_i)^2$ in $\R[C_1]\times\R[C_2]$, and we must
only see to it that $f_i(P)=g_i(P)$ for all $i$. This can easily be
done (Prop.~\ref{Prop:redcurvetransversal}). Thus, for example,
psd$=$sos in $\R[x,y]/(xy)$, which is the coordinate ring of two
intersecting lines. We show that either a curve can be build up
inductively from irreducible (or compact reducible) curves in this
simple way, or else there exists some obstacle that prevents
psd$=$sos. The main result, which completes Scheiderer's
classification of affine curves for which psd$=$sos holds, is the
following:

\begin{Thm*}
Let $C$ be an affine curve over $\R$, and let $C'$ be the union of all irreducible components of $C$ that do not admit any non-constant bounded polynomial function. Then psd$=$sos in $\R[C]$ if and only if the following conditions are satisfied:
\begin{enumerate}
\item All real singularities of $C$ are ordinary multiple points with independent tangents.
\item All intersection points of $C$ are real.
\item All irreducible components of $C'$ are non-singular and rational.
\item The configuration of the irreducible components of $C'$ contains no loops.
\end{enumerate}
\end{Thm*}

Here, an ordinary multiple point with independent tangents is just the proper higher-dimensional analogue of an ordinary double point. The last condition will be made precise in section \ref{Sec:SOSCurves}. Some examples:

\begin{enumerate}
\item Let $C=\{xy(1-x-y)=0\}$, three lines forming a triangle. This
constitutes the kind of loop that condition (4) of the theorem
forbids.  For a concrete example of a psd function that is not a sum
of squares in $\R[C]$, write $\R[C]\isom\bigl\{(f,g,h)\in \R[t]\times
\R[u]\times \R[v]\:\bigl|\: f(0)=g(1), f(1)=h(0), g(0)=h(1)\bigr\}$,
and put $f=2t-1$, $g=2u-1$, $h=2v-1$. Then $(f^2,g^2,h^2)$ is an
element of $\R[C]$ which is clearly psd. But if $f^2=\sum f_i^2$, with
all $f_i\neq 0$, then $f_i\bigl(\frac 12\bigr)=0$ and $\deg f_i=1$, so
$f_i=a_i f$ with constants $a_i\in\R$ for all $i$. The same is true
for $g$ and $h$. Therefore, if we had $(f^2,g^2,h^2)=\sum s_i^2$ in
$\R[C]$, then each $s_i$ would have to be given as $s_i=(a_if,b_i
g,c_i h)$ with constants $a_i,b_i,c_i\neq 0$. Such a triple can never
define a function on $C$, for if $a_if(0)$ and $b_ig(1)$ have the same
sign and $a_if(1)$ and $c_ih(0)$ have the same sign, then $b_ig(0)$
and $c_ih(1)$ will have different signs.

\item Consider the family of curves $C_a=\{(y-x^2)(y-a)=0\}$ for
$a\in\R$, the union of a parabola and a line. Again, psd$\neq$sos in
$\R[C_a]$ for any value of the parameter $a$, but for varying reasons:
If $a<0$, then the parabola and the line intersect at a pair of
distinct complex-conjugate points so that condition (2) of the theorem
is not met. For $a=0$, the intersection of the line and the parabola
in the origin is not an ordinary double point, violating condition
(1). And if $a>0$, the line and the parabola intersect at two distinct
real points which violates condition (4). (In all three cases, one can
argue directly in a similar way as in the previous example.)

\item For a positive example, let $C$ be the curve $\{(x^2+y^2-1)y=0\}$, a line intersecting a circle in two distinct real points. Since the real points of the circle $\{x^2+y^2=1\}$ are compact, condition (4) is empty, and psd$=$sos in $\R[C]$.
\end{enumerate}

A brief overview of the structure of this paper: After a few preliminaries in section \ref{sec:somefacts}, we introduce general techniques for dealing with sums of squares on reducible varieties, that are not peculiar to curves, in section \ref{Sec:Generalities}. But even if we assume that irreducible components intersect at only finitely many points, there is not much to be said here in complete generality. The most useful results are the basic Prop.~\ref{Prop:redcurvetransversal}, as well as Prop.~\ref{Thm:SOSOneCompactComp} which deals with a union of two subvarieties one of which is assumed compact. Section \ref{Sec:SOSCurves} is the longest, mostly devoted to the proof of the above theorem (Thm.~\ref{MainResult}) in several steps. In section \ref{Sec:POCurves}, we look at possible generalisations to the case of finitely generated preorderings in place of sums of squares. We give sufficient conditions for a reasonably large class of examples (Prop.~\ref{Prop:SatPORedCurv}), but many cases remain open. In section \ref{Sec:MP}, we briefly explain how our results can be applied to the moment problem of functional analysis.

\medskip

\emph{Acknowledgements:} This paper has developed from a part of my
PhD thesis \cite{MeineDiss}. I want to thank my advisor Claus
Scheiderer for introducing me to the questions considered here as well
as for his insights and suggestions. I am grateful to Michel
Coste for pointing out to me a simplification of the combinatorics in section
\ref{Sec:SOSCurves}. I also wish to thank the referee for valuable
remarks and suggestions.

\section{Preliminaries}\label{sec:somefacts}

In this section, we fix notations and briefly discuss some notions from real algebra and geometry, as well as some general facts concerning reducible algebraic varieties and curves. From algebraic geometry, we need only basic concepts and results, but because of the real ground field and a few non-reduced phenomena, it is convenient to work with schemes.  

\medskip
Let always $k$ be a field and $R$ a real closed field, for example the field of real numbers denoted by $\R$. 

\begin{Text} A \emph{variety} over $k$ is a reduced separated scheme
of finite type over $\Spec(k)$, not necessarily irreducible. A
\emph{curve} is a variety all of whose irreducible components have
(Krull-)dimension $1$. If $V$ is a variety over $k$ and $K/k$ a field
extension, we denote by $V(K)$ the set of $K$-valued points of $V$. We
frequently consider $V(k)$ as a subset of $V$ by identifying $V(k)$
with the set of points of $V$ with residue field $k$. We use the notation $k[V]$ for the coordinate ring $\sO_V(V)$ of an affine $k$-variety $V$.
\end{Text}

\begin{Text}
Let $V$ be a variety over $R$. Any closed point $P\in V$ has residue
field $\kappa(P)=R$ or $\kappa(P)=R(\sqrt{-1})$. In the first case,
the point $P$ is called \emph{real}, in the second \emph{non-real}.
The set $V(R)$ of real points is equipped with the semialgebraic
topology, induced by the ordering of $R$, which is the euclidean
topology if $R=\R$. Unless explicitly stated otherwise, topological
statements about subsets of $V(R)$ will refer to that topology.  If
$C$ is a non-singular curve over $R$, a \emph{divisor} on $C$ is any
finite $\Z$-linear combination of closed points. The \emph{degree} of
a divisor $\sum_{P\in C} n_P P$ is defined as $\sum n_P[\kappa(P):R]$;
in other words, non-real points are counted with multiplicity $2$. If
$C$ is irreducible and $P\in C$ is a closed point, we denote by ${\rm
  ord}_P$ the discrete valuation of the function field $R(C)$
corresponding to the order of vanishing at $P$. For a rational
function $f\in R(C)$, we write $\div_C(f)$ for the divisor $\sum_{P\in
  C} {\rm ord}_P(f)P$ of zeros and poles of $f$.
\end{Text}

\begin{Def}
Let $V$ be an affine $R$-variety, and let $V_1$, $V_2$ be closed subvarieties of $V$. Assume that $V_1$ and $V_2$ intersect at only finitely many points $P_1,\dots,P_r$ of $V$. We say the intersection of $V_1$ and $V_2$ is \emph{transversal} or that $V_1$ and $V_2$ intersect transversally if $\sI_{V_1}+\sI_{V_2}=\bigcap_{i=1}^r\fm_{P_i}$.
\end{Def}

Here, $I_W$ denotes the vanishing ideal of a subvariety $W$ in $R[V]$, and $\fm_P$ denotes the vanishing ideal of a point $P$. 

\begin{Lemma}\label{Lemma:FunctionsOnTransversalIntersection}
If $V$ is the union of two closed subvarieties $V_1$ and $V_2$ that intersect at finitely many points $P_1,\dots,P_r$, then the intersection of $V_1$ and $V_2$ is transversal if and only if the diagonal homomorphism 
\[
R[V]\into\bigl\{(f,g)\in R[V_1]\times R[V_2]\:|\: \forall i\in\{1,\dots,r\}\colon f(P_i)=g(P_i)\bigr\}
\]
given by $f\mapsto (f|_{V_1},f|_{V_2})$ is an isomorphism.
\end{Lemma}

\begin{proof}
Let $\phi$ be the map $R[V]\into R[V_1]\times R[V_2]$ given by $f\mapsto (f|_{V_1},f|_{V_2})$. We have $\sI_{V_1}\cap\sI_{V_2}=(0)$, because $V$ is the union of $V_1$ and $V_2$, hence $\phi$ is injective. The image of $\phi$ consists of all elements $(f|_{V_1},g|_{V_2})\in R[V_1]\times R[V_2]$, $f,g\in R[V]$, such that $f-g\in\sI_{V_1}+\sI_{V_2}$. On the other hand, $(f,g)\in R[V]\times R[V]$ satisfies $f(P_i)=g(P_i)$ for all $i\in\{1,\dots,r\}$ if and only if $f-g\in\bigcap_{i=1}^r\fm_{P_i}$. Thus the image of $\phi$ has the desired form if and only if $\sI_{V_1}+\sI_{V_2}=\bigcap_{i=1}^r\fm_{P_i}$ which proves the claim.
\end{proof}

\begin{Text} Let $C$ be a curve over $k$. 
A closed point $P\in C$ with residue field $K$ is called
an \emph{ordinary multiple point with independent tangents} if the
completed local ring $\wh{\sO}_{C,P}$ is isomorphic to
$K[[x_1,\dots,x_n]]/(x_ix_j\:|\:1\le i<j\le n)$ for some
$n$.\\
Every non-singular point on a curve is an ordinary multiple point with
independent tangents ($n=1$). A point on a plane curve is an ordinary multiple
point with independent tangents if and only if it is an ordinary
double point as defined in the introduction ($n=2$). 
\end{Text}

\begin{Lemma}\label{Lemma:OMPIT}
For any field $K$ and any $n\ge 1$, there is an isomorphism
\begin{align*}
& K[[x_1,\dots,x_n]]/(x_ix_j\:|\:1\le i< j\le n)\\
\isom\; &\bigl\{(f_1,\dots,f_n)\in K[[t_1]]\times\cdots\times K[[t_n]]\:\bigl|\:
f_1(0)=\cdots=f_n(0)\bigr\}.
\end{align*}
given by the the map 
$
\phi\colon K[[x_1,\dots,x_n]]/(x_ix_j\:|\:i,j)\longrightarrow \prod_i K[[t_i]]
$
that sends the class of $f\in K[[x_1,\dots,x_n]]$ to $\bigl(f(t_1,0,\dots,0),\dots,f(0,\dots,0,t_n)\bigr)$,
\end{Lemma}

\begin{proof}
One checks that $\phi$ is well-defined and injective. To see
that the image is as claimed, let $(f_1,\dots,f_n)$ be an element of the
right-hand side such that $f_1(0)=\cdots=f_n(0)$. Then
$\phi\bigl(f_1(x_1)+\cdots+f_n(x_n)-(n-1)f_1(0)\bigr)=(f_1,\dots,f_n)$.
\end{proof}

\begin{Cor}\label{Cor:FnOnRedCurves}
Let $C$ be a curve with irreducible components $C_1,\dots,C_m$ over a field $k$. Assume that all singularities of $C_1,\dots,C_m$ are ordinary multiple points with independent tangents. The following are equivalent:

\begin{enumerate}
\item All singularities of $C$ are ordinary multiple points with
independent tangents.
\item  The closed subcurves $C_i$ and $C_i'=\bigcup_{j\neq i}
C_j$ intersect transversally for every $i=1,\dots,m$.
\item The diagonal homomorphism $k[C]\into\prod_{i=1}^m k[C_i]$
given by $f\mapsto (f|_{C_i})_{i=1,\dots,m}$ induces an isomorphism between $k[C]$ and
\[
\bigl\{(f_i)\in\prod_ik[C_i]\:\bigl|\: f_i(P)=f_j(P)\text{
for all }P\in C_i\cap C_j,\ i,j=1,\dots,m\bigr\}.
\]
(where $C_i\cap C_j$ is the usual, set-theoretic intersection).
\end{enumerate}
\end{Cor}

\begin{proof}
The equivalence of (2) and (3) is immediate from Lemma \ref{Lemma:FunctionsOnTransversalIntersection}. Assume that (1) holds and write $\phi$ for the diagonal homomorphism $k[C]\into\prod_{i=1}^m k[C_i]$. For any closed point $P\in C$, consider the induced map $\phi_P\colon{\wh\sO}_{C,P}\into\prod_i{\wh\sO}_{C_i,P}$ on completed local rings (note that ${\wh\sO}_{C_i,P}=0$ if $P\notin C_i$). Since $P$ is an ordinary multiple point with independent tangents, it follows from Lemma \ref{Lemma:OMPIT} that $\phi_P$ is an isomorphism. So $\phi$ is an isomorphism after completion with respect to any maximal ideal of $k[C]$, hence it is an isomorphism (see for example Bourbaki \cite{MR0171800}, \S3, no.~5, Cor.~5). 

Conversely, if (3) holds and $P\in C$ is any closed point, then $P$ is an ordinary multiple point with independent tangents on each irreducible component of $C$ passing through $P$ by assumption, hence by Lemma \ref{Lemma:OMPIT}, $P$ must also be an ordinary multiple point with independent tangents on $C$.
\end{proof}

\begin{Text}
A field is called \emph{real} if it possesses an ordering, i.e.~a linear order of the set $K$ that respects addition and multiplication. A valuation $v\colon K^{\times}\into\Gamma$ of a field $K$ is called \emph{real} if its residue field is real. The following simple lemma is of fundamental importance to the study of sums of squares on real varieties. It generalises the observation that ``leading terms cannot cancel'' in a sum of squares of real polynomials.

\end{Text}

\begin{Lemma}\label{Cor:ValuationSOS}
Let $K$ be a field equipped with a real valuation $v$, and let
$a_1,\dots,a_r\in K$. Then
$v(a_1^2+\cdots+a_r^2)=2\min\{v(a_1),\dots,v(a_r)\}$.
\end{Lemma}

\begin{proof}
By the Baer-Krull theorem (see for example Bochnak, Coste, and Roy
\cite{MR1659509}, Thm.~10.1.10), there exists an ordering of $K$ that
is compatible with the valuation $v$. Then $\sum_j a_j^2\ge a_i^2$ in $K$ implies $v(\sum_j a_j^2)\le 2v(a_i)$ for all $i$. The opposite inequality holds trivially.
\end{proof}

\begin{Text}
An irreducible variety $V$ over $R$ is called \emph{real} if the following equivalent conditions are satisfied: (1) $V(R)$ is Zariski-dense in $V$. (2) $V(R)$ contains a non-singular point of $V$. (3) The semialgebraic dimension of $V(R)$ coincides with the (Krull-) dimension of $V$. (4) The function field $R(V)$ is real (see \cite{MeineDiss}, Prop.~A.1 or Bochnak, Coste, and Roy \cite{MR1659509}, section 7.6). A reducible variety over $R$ is called real if all its irreducible components are real.
\end{Text}

\begin{Text}\label{Def:RBF}
Let $V$ be an affine $R$-variety, and let $S$ be a semialgebraic subset of $V(R)$. We write
\[
B_V(S)=\bigl\{f\in R[V]\:|\: \exists\lambda\in R\;\forall x\in S\colon |f(x)|\le\lambda\bigr\}
\]
for the ring of bounded functions on $S$. Its size can be seen as a measure for the ``compactness'' of $S$. Mostly, we will be interested in the case $S=V(R)$ for which we use the notation
\[
B(V)=B_V(V(\R)).
\]
We collect what we will need about rings of bounded functions on
curves in the following lemma:
\end{Text}

\begin{Lemma}\label{Lemma:VirtualCompactness}
Let $C$ be an affine curve with irreducible components $C_1,\dots,C_m$ over
$R$. Let $S$ be a semialgebraic subset of $C(R)$ and write $S_i=S\cap
C_i(R)$.

\begin{enumerate}
\item There exists an open-dense embedding of $C$ into a projective curve $X$ such that the finitely many points of $X\setminus C$ are non-singular. The embedding $C\injects X$ is unique up to isomorphism. 
\item Let $T$ be the union of all points of $X\setminus C$ that are
either non-real, or real but not contained in the closure of
$S$ in $X(R)$. Put $\wt C=X\setminus T$. Then $B_C(S)\isom\sO_{\wt{C}}(\wt C)$.
\item If $C$ is irreducible, then $B_C(S)\neq R$ if and only if
$T\neq\emptyset$. 
\item If $B_{C_i}(S_i)\neq R$ for all $i=1,\dots,m$, then $\wt C\isom\Spec(B_C(S))$. 
\item If $B_{C_j}(S_j)\neq R$ for some $j\in\{1,\dots,m\}$, there
exists a non-constant element $f\in B_C(S)$ such that $f|_{C_i}=0$ for
all $i\neq j$. In particular, if $C$ is connected, then $B_C(S)=R$ if
and only if $B_{C_i}(S_i)=R$ for all $i=1,\dots,m$. 
\item If $h\in B_C(S)$ vanishes at all points of $\wt C\setminus C$, then for every $f\in R[C]$ there exists $N\ge 0$ such that $h^Nf\in B_C(S)$.
\item If $B_C(S)=R$ and $f_1,\dots,f_r\in R[C]$ are such that $f_1^2+\cdots+f_r^2\in R$, then $f_1,\dots,f_r\in R$.
\end{enumerate}
\end{Lemma}

\begin{Remark}
We will mostly use this lemma in the global case $S=C(R)$. Note that,
since the points of $X\setminus C$ in (1) are non-singular by
definition, $C(R)$ is dense in $X(R)$. So if $S=C(R)$, then
$T$ in (2) consists exactly of the non-real points of $X\setminus
C$; in particular, $\wt C(R)=X(R)$ is semialgebraically compact.
\end{Remark}

\begin{proof}
(1) For the existence of $C\injects X$, start with any embedding of
$C$ into affine space, take the closure $X_0$ in the corresponding
projective space, and apply resolution of singularities for curves to
the finitely many points of $X_0\setminus C$. If $C\injects X_1$,
$C\injects X_2$ are two such embeddings, the identity map $C\into C$
induces a birational morphism $X_1\ratmap X_2$ which is an isomorphism
since $X_1$ and $X_2$ are projective and all points of $X_1\setminus C$ and $X_2\setminus C$ are non-singular (see for example Hartshorne \cite{MR0463157}, Prop.~6.8).

(2) Since $X$ is projective, $X(R)$ is semialgebraically
compact. Hence so is the closure $\ol S$ of $S$ in $X(R)$. Therefore,
$\ol S\subset\wt C(R)$ implies $\sO_{\wt C}(\wt C)\subset B_C(S)$. For
the converse, if $f\in R[C]$, $f\notin\sO_{\wt{C}}(\wt C)$, then $f$
has a pole at a point of $\ol{S}\setminus S$. It is easy to see that $f$ cannot be bounded on $S$ (see also \cite{MeineDiss}, Lemma 1.8).

(3) $B_C(S)=R$ if and only if there does not exist a rational function
$f\in R(C)$ with poles only at points of $T$. This only happens if $T=\emptyset$, by the Riemann-Roch theorem.

(4) In general, if $V$ is any variety, then $V\isom\Spec(\sO_V(V))$ if
and only if $V$ is affine. But an irreducible curve is either affine
or projective, and the hypothesis $B_{C_i}(S_i)\neq R$ for
all $i=1,\dots,m$ implies that none of the $\wt C_i$ is
projective, by (3) (where $\wt C_i$ is defined for $C_i$ as $\wt C$ for $C$ in
(2)).

(5) Let $j\in\{1,\dots,m\}$ with $B_{C_j}(S_j)\neq R$. Write
$C'=\bigcup_{i\neq j} C_i$. To find $f$ as in the claim, let $J$ be
the vanishing ideal of $C_j$ in $R[C]$, so that $R[C_j]=R[C]/J$, and
let $I$ be the vanishing ideal of $C'$ in $R[C]$. The ideal
$I_j=(I+J)/J$ of $R[C_j]$ is non-zero, since $C_j\nsubseteq
C'$. Therefore, the residue class ring $R[C_j]/I_j$ is
zero-dimensional, thus it is a finite-dimensional $R$-vector
space. On the other hand, $B_{C_j}(S_j)$ is isomorphic to the
coordinate ring of an affine curve by (2) and (4) and is therefore an
infinite-dimensional subspace of $R[C_j]$. It follows that $I_j\cap
B_{C_j}(S_j)$ is also infinite-dimensional. Hence there exists $f\in I$ such that
$f|_{C_j}\in B_{C_j}(S_j)\setminus R$, and any such $f$ will do what
we want.

If $C$ is connected, a non-constant function on $C$ that is
bounded on $S$ must be non-constant on some $C_i$ and bounded on
$S_i$, so $B_{C_i}(S_i)=R$ for all $i=1,\dots,m$ implies
$B_C(S)=R$. 

(6) We have $B_C(S)=\bigcap B_{C}(S_i)$. Therefore, we may assume that $C$ is
irreducible. Write $\wt C\setminus C=\{P_1,\dots,P_r\}$. Statement (2) says that $0\neq f\in R[C]$ lies in $B(C)$ if and only if ${\rm ord}_{P_i}(f)\ge 0$ for all $i=1,\dots,r$. We may assume $h\neq 0$. Since ${\rm ord}_{P_i}(h)>0$ for all $i=1,\dots,r$ by hypothesis, it follows that there exists $N\ge 0$ such that ${\rm ord}_{P_i}(h^Nf)=N\cdot {\rm ord}_{P_i}(h)+{\rm ord}_{P_i}(f)\ge 0$ holds for all $i=1,\dots,r$.

(7) By (5), $B_C(S)=R$ implies $B_{C_i}(S_i)=R$ for all irreducible components $C_i$ of $C$. Thus we may assume that $C$ is irreducible. Let $P\in X\setminus C$. Then $P$ must be real by (3), hence the corresponding valuation ${\rm ord}_P$ is a real valuation of the function field $R(C)$. Thus if $f_1^2+\cdots+f_r^2\in R$, then ${\rm ord}_P(f_i)=0$ for all $i=1,\dots,r$ by Lemma~\ref{Cor:ValuationSOS}. It follows that the $f_i$ have no poles on the complete curve $X$, i.e.~they are contained in the intersection of all valuation rings of $R(C)$, which is $R$. 
\end{proof}

\begin{Def}\label{Def:VC}
The points of $X\setminus C$ with $C\injects X$ as in the lemma, are called the \emph{points at infinity} of $C$. The set of real points $C(R)$ (or, loosely speaking, the curve $C$ itself) is called \emph{virtually compact} if $B(C_i)\neq R$ holds for every irreducible component $C_i$ of $C$ or, equivalently, if every irreducible component of $C$ has a non-real point at infinity (see Def.~4.8.~in Scheiderer \cite{MR2020709}).
\end{Def}

\begin{Text}\label{Def:PO}
We briefly fix notations and terminology for preorderings and
semialgebraic sets that will be used in section \ref{Sec:POCurves}:
Let $A$ be a ring. A \emph{preordering} of $A$ is a subset $T$ of $A$
that is closed under addition and multiplication and contains all
squares of elements of $A$. Given a subset $\sH$ of $A$, the \emph{preordering generated by $\sH$} is the intersection of all preorderings of $A$ containing $\sH$ and is denoted by $\PO_A(\sH)$ or just $\PO(\sH)$. If $\sH=\{h_1,\dots,h_r\}$ is finite, the generated preordering has a simple explicit description:
\[
\PO_A(\sH)=\left\{\sum_{i\in\{0,1\}^r} s_i \ul h^i\;\bigl|\; s_i\in\sum A^2\right\},
\]
where we use the notation $\ul h^i=h_1^{i_1}\cdots h_r^{i_r}$. A preordering $T$ of $A$ is called finitely generated if there exists a finite subset $\sH$ of $A$ such that $T=\PO(\sH)$.

If $V$ is an affine $R$-variety with coordinate ring $R[V]$ and $S$ a subset of $V(R)$, the psd-cone of $S$ 
\[
\sP_V(S)=\bigl\{f\in R[V]\;|\; \forall x\in S\colon f(x)\ge 0\bigr\}
\]
is a preordering of $R[V]$. Conversely, given a subset $\sH$ of $R[V]$, we write
\[
\sS(\sH)=\bigl\{x\in V(R)\:|\: \forall h\in\sH\colon h(x)\ge 0\bigr\}.
\]
A subset $S$ of $V(R)$ is called \emph{basic closed} if there exists some finite subset $\sH$ of $R[V]$ such that $S=\sS(\sH)$. A finitely generated preordering $T$ of $R[V]$ is called \emph{saturated} if $T=\PO(\sS(T))$.

If $T$ is a preordering of $R[V]$ and $Z$ a closed subvariety of $V$ with vanishing ideal $\sI_Z\subset R[V]$, we denote the induced preordering $(T+\sI_Z)/\sI_Z$ of $R[Z]=R[V]/\sI_Z$ by $T|_Z$.
\end{Text}

\section{Sums of squares on reducible varieties: Generalities}\label{Sec:Generalities}

Let always $R$ be a real closed field, and let $V$ be an affine
$R$-variety with coordinate ring $R[V]$. We say that \emph{psd$=$sos in $R[V]$} if for every $f\in R[V]$ with $f(P)\ge 0$ for all $P\in V(R)$ there exist $f_1,\dots,f_n\in R[V]$ such that $f=f_1^2+\cdots+f_n^2$. This is what is known in the irreducible case:

\begin{Text}\label{Rem:KnownResults}
\begin{enumerate}
\item If $V$ is zero-dimensional, then $R[V]$ is a direct product of copies of $R$ and $R(\sqrt{-1})$, so that psd$=$sos in $R[V]$. 
\item If $R=\R$ and $V$ is an irreducible affine curve, Scheiderer has
given necessary and sufficient conditions for psd$=$sos in $\R[V]$; see Cor.~\ref{Cor:SOSirrcurvVC} and Thm.~\ref{Thm:SOSirrCurvesNVC} below.
\item If $R=\R$ and $V$ is a non-singular irreducible affine surface over $\R$ such that $V(\R)$ is compact, then psd$=$sos in $\R[V]$ (Scheiderer \cite{MR2223624}, Cor.~3.4).
\item There exist examples of non-singular irreducible affine surfaces $V$ over $\R$ such that $V(\R)$ is not compact and psd$=$sos in $\R[V]$ (Scheiderer \cite{MR2223624}, Remark 3.15, and \cite{MeineDiss}, section 4.4). 
\item If $V$ is irreducible and not real, then it is not hard to show
that psd$=$sos in $R[V]$ if and only if $V(R)=\emptyset$; see also Cor.~\ref{Prop:Nonrealcomponents} below. 
\end{enumerate}
\end{Text}

We now turn to reducible varieties. To avoid confusion, note that we
will be looking at three kinds of components of $V$:
\emph{Irreducible} components of $V$ (for the Zariski-topology);
\emph{connected} components of $V$ (again for the Zariski-topology);
and occasionally connected components of $V(R)$ (for the semialgebraic
topology which is the euclidean topology for $R=\R$).  Any irreducible
component is connected, and any connected component is a union of
irreducible components. But even if $V$ is irreducible, $V(R)$ need
not be connected.

\medskip

\begin{Question*}
Let $V_1,\dots,V_m$ be the irreducible components of $V$, and assume
that psd$=$sos in $R[V_1],\dots,R[V_m]$. Under what circumstances can we conclude that psd$=$sos in $R[V]$?
\end{Question*}

\begin{Remark*}
It is not \emph{a priori} clear that psd$=$sos in $R[V]$ implies
psd$=$sos in each $R[V_i]$, since a psd function on $V_i(R)$ need not
extend to a psd function on $V(R)$. We will later see that this
implication does indeed hold for curves
(Prop.~\ref{Prop:psdsos_passestoirrcomp}). It would seem that it
should be true in general, but I do not know of a way to prove this.
\end{Remark*}

Note that if $V$ has connected components
$W_1,\dots,W_l$, then $R[V]\isom R[W_1]\times\cdots\times R[W_l]$ and
psd$=$sos in $R[V]$ if and only if psd$=$sos in all
$R[W_1],\dots,R[W_l]$. Thus we can always assume that $V$ is
connected, and the interesting data is how the irreducible components
$V_1,\dots,V_m$ intersect. We begin with the simplest conceivable
case:

\begin{Prop}\label{Prop:redcurvetransversal}
Let $V$ be an affine $R$-variety that is the union of two closed
subvarieties $V_1$ and $V_2$ intersecting transversally in a single
real point, and let $f\in R[V]$. If $f|_{V_1}$ and
$f|_{V_2}$ are sums of squares in $R[V_1]$ resp.~$R[V_2]$, then $f$ is a sum of squares in $R[V]$.
\end{Prop}

\begin{proof}
Write $V_1\cap V_2=\{P\}\in V(R)$. Using Lemma
\ref{Lemma:FunctionsOnTransversalIntersection}, we can identify $R[V]$
with the subring of $R[V_1]\times R[V_2]$ consisting of all pairs
$(f,g)$ such that $f(P)=g(P)$. Let $(f,g)\in R[V]$ be such that
$f=\sum_{i=1}^n f_i^2$, $g=\sum_{i=1}^n g_i^2$ with $f_1,\dots,f_n\in
R[V_1]$ and $g_1,\dots,g_n\in R[V_2]$ (not necessarily $\neq 0$). The
vectors $v=\bigl(f_1(P),\dots,f_n(P)\bigr)^t$ and
$w=\bigl(g_1(Q),\dots,g_n(Q)\bigr)^t$ in $R^n$ have the same euclidean
length since $f(P)=g(Q)$. Therefore, there exists an orthogonal matrix
$B\in O_n(R)$ such that $Bv=w$. Put $(\wt f_1,\dots,\wt f_n)^t=B\cdot
(f_1,\dots,f_n)^t$, then $f=\sum\wt f_i^2$ and $(\wt f_i,g_i)\in R[V]$
for all $i$, hence $(f,g)=\sum_{i=1}^n (\wt f_i,g_i)^2$.
\end{proof}

\begin{Examples}\label{Ex:TwoCurvesPSDSOS}
\begin{enumerate}
\item Let $C$ be the plane curve $\{xy=0\}$. Then psd$=$sos in $R[C]$ by the proposition.
\item More generally, let $V_1$ and $V_2$ be any two affine $R$-varieties with $V_i(R)\neq\emptyset$. Fix points $P\in V_1(R)$ and $Q\in V_2(R)$, take the ring $A=\{(f,g)\in R[V_1]\times R[V_2]\:|\: f(P)=g(Q)\}$, and let $V$ be the affine variety $\Spec(A)$. (To see that $A$ is a finitely generated $R$-algebra, choose generators $x_1,\dots,x_n$ of $R[V_1]$ and $y_1,\dots,y_m$ of $R[V_2]$ such that $x_i(P)=0$ and $y_i(Q)=0$ for all $i$. Then $A$ is generated by the elements $(x_i,0),(0,y_i)$ and $(1,1)$.) The variety $V$ is $V_1$ and $V_2$ glued transversally along $P$ and $Q$. If psd$=$sos in $R[V_1]$ and $R[V_2]$, it also holds in $R[V]$. A simple example would be given by two $2$-dimensional spheres over $\R$ intersecting transversally.
\end{enumerate}
\end{Examples}

If $V_1$ and $V_2$
intersect at more than one point, the statement of
\ref{Prop:redcurvetransversal} becomes false in general, as we shall
see later. However, we can still say something if $R=\R$ and $V_1(\R)$
is compact. For the case of curves, we will also allow the slightly
weaker condition of virtual compactness: Recall from \ref{Def:RBF}
that $B(V)$ denotes the ring of bounded functions on an affine variety
$V$; if $C$ is a curve over $\R$, then $C(\R)$ is called virtually
compact if every irreducible component of $C$ admits a non-constant
bounded function; see \ref{Def:VC}. 

We will need the existence of a certain kind of polynomial partition of unity adapted to points in the sense of the following Lemma. It is an easy consequence of a basic topological lemma due to Kuhlmann, Marshall, and Schwartz. 

\begin{Lemma}
Let $V$ be an affine $\R$-variety such that $V(\R)$ is compact or, if $V$ is a curve, virtually compact. Given finitely many distinct points $P_1,\dots,P_r\in V(\R)$, there exist $h_1,\dots,h_r\in B(V)$ with the following properties:
\begin{enumerate}
\item $\sum_{i=1}^rh_i=1$;
\item $h_i\ge 0$ on $V(\R)$;
\item $h_i(P_j)=0$ for all $i\neq j$ (hence $h_i(P_i)=1$).
\end{enumerate}
\end{Lemma}

\begin{proof}
Assume $r\ge 2$, the case $r=1$ being trivial. Let $W=\Spec(B(V))$. If $V(\R)$ is compact, then $B(V)=\R[V]$ and $W=V$. If $V$ is a curve and $V(\R)$ is virtually compact, then the canonical morphism $V\into W$ induced by the inclusion $B(V)\subset\R[V]$ is an embedding of affine curves and $W(\R)$ is compact (Lemma \ref{Lemma:VirtualCompactness}). It then suffices to prove the Lemma for $W$. It is therefore not restrictive to assume that $V(\R)$ is compact.

Choose elements $g_1,\dots,g_r\in\R[V]$ such that $g_i(P_j)=0$ for all
$i\neq j$ and $g_i(P_i)\neq 0$. Since
$(g_1^2+\cdots+g_{r-1}^2)(P_i)=0$ if and only if $i=r$, the Nullstellensatz gives an identity
$a(g_1^2+\cdots+g_{r-1}^2)+g=1$ with $a,g\in\R[V]$ and $g(P_i)=0$ for
$i=1,\dots,r-1$. It follows that $g_1^2+\cdots+g_{r-1}^2$ and $g$ have
no common (complex) zeros on $V$. Hence again by the Nullstellensatz,
there is an identity $s(g_1^2+\cdots+g_{r-1}^2)+tg^2=1$,
$s,t\in\R[V]$. Now since $V(\R)$ is compact, Lemma 2.1 in
\cite{MR2174483} states that we can find such $s,t$ with the
additional property that $s$ and $t$ are strictly positive on
$V(\R)$. Thus $h_i=sg_i^2$ for $i\in\{1,\dots,r-1\}$ and $h_r=tg^2$
will do what we want.
\end{proof}

\begin{Cor}\label{Cor:PoU}
Let $V$ be an affine $\R$-variety such that $V(\R)$ is compact or, if
$V$ is a curve, virtually compact. Given finitely many distinct points
$P_1,\dots,P_r\in V(\R)$ and real numbers $a_1,\dots,a_r\in [-1,1]$,
there exists a function $g\in\R[V]$ such that $g(P_i)=a_i$ for all
$i\in\{1,\dots,r\}$ and $|g(x)|\le 1$ for all $x\in V(\R)$. If
$a_1,\dots,a_r$ are all non-negative, one can find such a $g$ that is non-negative on $V(\R)$, as well.
\end{Cor}

\begin{proof}
Take $h_1,\dots,h_r$ as in the lemma and put $g=\sum_{i=1}^r
a_ih_i$. Then $g(P_i)=a_i$ and $|g(x)|\le\sum_i |a_i|h_i(x)\le\sum_i
h_i(x)=1$ for all $x\in V(\R)$. If $a_1,\dots,a_r$ are non-negative, take
$h\in\R[V]$ such that $h(P_i)=\sqrt{a_i}$ for all $i=1,\dots,r$ and $|h|\le 1$ on $V(\R)$, and put $g=h^2$.
\end{proof}

\begin{Lemma}
Let $V$ be an affine $\R$-variety such that $V(\R)$ is compact or, if
$V$ is a curve, virtually compact. Assume that psd$=$sos in
$\R[V]$. Let $P_1,\dots,P_r\in V(\R)$ be distinct points, and let
$f\in\R[V]$ be psd. Given $m\ge 1$ and vectors
$a^{(1)},\dots,a^{(r)}\in\R^m$ such that $||a^{(i)}||=f(P_i)$ for all
$i\in\{1,\dots,r\}$, there exist $n\ge m$ and $f_1,\dots,f_n\in\R[V]$
with the following properties:
\begin{enumerate}
\item $f=\sum_{i=1}^n f_i^2$;
\item $f_j(P_i)=a_j^{(i)}$ for all $i=1,\dots,r$, $1\le j\le m$;
\item $f_j(P_i)=0$ for all $i=1,\dots,r$, $m+1\le j\le n$.
\end{enumerate}
\end{Lemma}

\begin{proof}
Note first that property (3) is automatic if (1) and (2) are
satisfied. We first show that there exists $\wt f\in B(V)$ such that
$0\le\wt f\le f$ on $V(\R)$ and such that $\wt f(P_i)=f(P_i)$ for all
$i=1,\dots,r$. If $V(\R)$ is compact, we can take $\wt f=f$. If $V$ is
a curve and $V(\R)$ is virtually compact, put $W=\Spec(B(V))$ and
consider $V$ as a subcurve of $W$ via the embedding $V\injects W$
induced by the inclusion $B(V)\subset\R[V]$ (Lemma
\ref{Lemma:VirtualCompactness}). Let $Q_1,\dots,Q_s\in W(\R)$ be the
finitely many points of $W\setminus V$. By Cor.~\ref{Cor:PoU}, there
exists $h\in B(V)$ such that $h(Q_i)=0$ for all $i=1,\dots,s$,
$h(P_i)=1$ for all $i=1,\dots,r$ and $0\le h\le 1$ on $W(\R)$. Then
$h^Nf\in B(V)$ for sufficiently large $N$ (by Lemma
\ref{Lemma:VirtualCompactness} (6)), and we can take $\wt f=h^Nf$ for
such an $N$.

We have $f-\wt f\ge 0$ on $V(\R)$ and $(f-\wt f)(P_i)=0$ for all
$i=1,\dots,r$. Since psd$=$sos in $\R[V]$, we can write $f=h_1^2+\cdots+h_k^2+\wt f$ with $h_1,\dots,h_k\in\R[V]$ and $h_j(P_i)=0$ for all $i=1,\dots,r$ and
$j=1,\dots,k$. It therefore suffices to prove the claim for $\wt
f$. 

Hence we may assume that $f$ is bounded on $V(\R)$. We may further
assume that $f\neq 0$ and that $f\le 1$, upon replacing $f$ by $f\cdot\sup_{x\in
V(\R)}\{f(x)\}^{-1}$. We define $f_1,\dots,f_m$
recursively as follows: Apply Cor.~\ref{Cor:PoU} and choose
$g_1\in\R[V]$ with $|g_1|\le 1$ on $V(\R)$ and such that $g_1(P_i)=0$
if $f(P_i)=0$ and $g_1(P_i)=a^{(i)}_1\cdot f(P_i)^{-1}$ otherwise (for
all $i=1,\dots,r$). Put $f_1=g_1 f$, then $f_1(P_i)=a_1^{(i)}$ (note
that $f(P_i)=0$ implies $a_1^{(i)}=0$) and $f_1^2\le f$. Now since
$f-f_1^2\ge 0$, we can apply Cor.~\ref{Cor:PoU} again and choose
$f_2\in\R[V]$ such that $f_2(P_i)=a_2^{(i)}$ and $f_2^2\le
f-f_1^2$. Continuing recursively in this manner will produce
$f_1,\dots,f_m$ satisfying property (2) and such that $f-f_1^2-\cdots-
f_m^2\ge 0$ holds on $V(\R)$. Since psd$=$sos in $\R[V]$, there exist
$n\ge m$ and elements $f_{m+1},\dots,f_n\in \R[V]$ such that
$f=\sum_{i=1}^n f_i^2$.
\end{proof}

\begin{Prop}\label{Thm:SOSOneCompactComp}
Let $V$ be an affine $\R$-variety, and assume that $V$ is the union of
two closed subvarieties $V_1$ and $V_2$ with the following properties:

\begin{enumerate}
\item psd$=$sos in $\R[V_1]$;
\item $V_1(\R)$ is compact, or $V_1$ is a curve with $V_1(\R)$ virtually compact;
\item The intersection of $V_1$ and $V_2$ is finite and transversal;
\item All points of $V_1\cap V_2$ are real.
\end{enumerate}
Then every psd function $f\in\R[V]$ such that $f|_{V_2}$ is a sum of squares in
$\R[V_2]$ is a sum of squares in $\R[V]$.
\end{Prop}

\begin{proof}
Write $V_1\cap V_2=\{P_1,\dots,P_r\}\subset V(\R)$, $r\ge 0$. Using Lemma \ref{Lemma:FunctionsOnTransversalIntersection}, we can
identify $\R[V]$ with the subring of $\R[V_1]\times\R[V_2]$ consisting
of all pairs $(f,g)$ such that $f(P_i)=g(P_i)$ for all
$i=1,\dots,r$. Let
$F\in\R[V]$ be non-negative on $V(\R)$ and such that
$g=F|_{V_2}=g_1^2+\cdots+g_m^2$ for some $m\ge 1$,
$g_1,\dots,g_m\in\R[V_2]$. 

By the preceding lemma (applied with $a^{(i)}_j=g_j(P_i)$), there
exist $n\ge m$ and elements $f_1,\dots,f_n\in\R[V_1]$ such that
$f=F|_{V_1}=\sum_{j=1}^n f_j^2$, $f_j(P_i)=g_j(P_i)$ for all $j\le m$,
and $f_j(P_i)=0$ for all $j>m$ ($i=1,\dots,r$). It follows that
$F_j=(f_j,g_j)$ for $j=1,\dots,m$ and $F_j=(f_j,0)$ for
$j=m+1,\dots,n$ are elements of $\R[V]$, and that
$F=(f,g)=\sum_{j=1}^n F_j^2$.
\end{proof}

We will mostly apply this proposition in the proof of our main result on curves (Thm.~\ref{MainResult}). As mentioned above, the condition that psd$=$sos in $\R[V_1]$ can only be satisfied if $V_1$ has dimension at most
$2$. But even in the two-dimensional case, hypothesis (3) is really
too restrictive for the proposition to be of much use. One can produce examples though: 

\begin{Example}
Let $V_1=\{x^2+y^2+z^2=1\}$, $V_2=\{x=y=0\}$, $V=V_1\cup V_2$, a sphere and a line intersecting transversally in affine three-space with coordinates $x,y,z$ over $\R$. Then psd$=$sos holds in $\R[V_1]$ by Scheiderer's results and in $\R[V_2]$ (elementary), thus it also holds in $\R[V]$.
\end{Example}

\noindent We conclude this section with an observation in the non-real case: 

\begin{Lemma}[Scheiderer \cite{MR1675230}, Lemma 6.3]
Let $A$ be a connected noetherian ring with $\Sper A\neq\emptyset$, and suppose that $A$ is not real reduced. Then there exists $f\in A$ that vanishes identically on $\Sper A$ but is not a sum of squares in $A$.
\end{Lemma}

\begin{Cor}\label{Prop:Nonrealcomponents}
Let $V$ be an affine $R$-variety and let $V_r$ be the union of all real irreducible components, $V_{nr}$ the union of all non-real irreducible components of $V$. Then psd$=$sos in $R[V]$ if and only if $V_{nr}(R)=\emptyset$, $V_r\cap V_{nr}=\emptyset$, and psd$=$sos in $R[V_r]$. 
\end{Cor}

\begin{proof}
Both $V_{nr}(R)\neq\emptyset$ or $V_r\cap V_{nr}\neq\emptyset$ imply that there exists a connected component $V'$ of $V$ such that $V'$ is not real but $V'(R)\neq\emptyset$, so that psd$\neq$sos in $R[V']$ by the lemma. As noted above, this implies psd$\neq$sos in $R[V]$. Conversely, if $V_{nr}(R)=\emptyset$ and $V_r\cap V_{nr}=\emptyset$, then every element of $R[V_{nr}]$ is a sum of squares in $R[V_{nr}]$ by the real Nullstellensatz and $R[V]\isom R[V_r]\times R[V_{nr}]$, so that psd$=$sos in $R[V_r]$ is necessary and sufficient for psd$=$sos in $R[V]$.
\end{proof}

\section{Sums of squares on curves}\label{Sec:SOSCurves}

Let $C$ be an affine curve over $R$ with irreducible components $C_1,\dots,C_m$. We begin by showing that psd$=$sos in $R[C]$ implies psd$=$sos in $R[C_1],\dots,R[C_m]$. We will need the following

\begin{Prop}[Scheiderer \cite{MR2020709}, Cor.~4.22]\label{Prop:BadRealSing}
Let $C$ be an affine curve over $R$. If $C$ has a real singular point
that is not an ordinary multiple point with independent tangents, then
psd$\neq$sos in $R[C]$.
\end{Prop}

\begin{Lemma}\label{Lem:LiftingOnCurves}
Let $C$ be an affine curve over $R$ all of whose real intersection
points are ordinary multiple points with independent tangents. Let
$C'\subset C$ be a closed subcurve of $C$. Then every psd function on $C'$
can be extended to a psd function on $C$, i.e.~every psd element of
$R[C']$ is the image of a psd element of $R[C]$ under the restriction map $R[C]\into R[C']$.
\end{Lemma}

\begin{proof}
Let $f$ be a psd function on $C'$, and let $D$ be the union of all
irreducible components of $C$ not contained in $C'$. Let $Z=C'\cap D$
be the scheme-theoretic intersection. We have to find a psd function
on $D$ that agrees with $f$ on the closed subscheme $Z$. If $Z$ is
supported on $r$ real points and $s$ non-real points, then the assumption on intersection points implies that the ring of
regular functions $R[Z]$ of $Z$ is a direct product
\[
R[Z]\isom \underbrace{R\times\cdots\times R}_{r}\times
A_1\times\cdots\times A_s
\]
with $R(\sqrt{-1})\subset A_i$ for all $i\in\{1,\dots,s\}$. This implies $\sum A_i^2=A_i$ for all $i$, and since $f$
takes only non-negative values at the real points of $Z$, it follows
that the class of $f$ in $R[Z]$ is a sum of squares in
$R[Z]$. Therefore, it can be lifted to a sum of squares in $R[D]$
which, in particular, is psd on $D$.
\end{proof}

\begin{Example*}
The condition on intersection points cannot be dropped in general: If
$C=\{y(y^2-x^3)=0\}$, the function $x$ is psd on
$C'=\{y^2-x^3=0\}$ but cannot be extended to a psd
function on $C$.
\end{Example*}

\begin{Prop}\label{Prop:psdsos_passestoirrcomp}
Let $C$ be an affine curve over $R$. If psd$=$sos in $R[C]$,
then psd$=$sos in $R[C']$ for every closed subcurve $C'$ of
$C$.
\end{Prop}

\begin{proof}
If all real singularities of $C$ are ordinary multiple points with
independent tangents, then every psd function $f$ on $C'$ can be
extended to a psd function $g$ on $C$ by Lemma
\ref{Lem:LiftingOnCurves}. Then $g$ is a sum of squares in $R[C]$ by
hypothesis, hence $f$ is a sum of squares in $R[C']$. If $C$ has a
real singular point that is not an ordinary multiple point with
independent tangents, then psd$\neq $sos in $R[C]$ by
Prop.~\ref{Prop:BadRealSing}, so the statement is empty.
\end{proof}

\begin{Prop}\label{Prop:redcurvenonrealintersection}
Let $C$ be a real curve over $R$, and assume that $C$ has a non-real
intersection point. Then psd$\neq$sos in $R[C]$.
\end{Prop}

\begin{proof}
Let $C_1$ and $C_2$ be two distinct irreducible components of $C$ that
intersect at a non-real point. By
Prop.~\ref{Prop:psdsos_passestoirrcomp}, it suffices to show that
psd$\neq$sos in $R[C_1\cup C_2]$. We may therefore assume that
$C=C_1\cup C_2$.

Let $I_j=\sI_{R[C]}(C_j)$ be the vanishing ideal of $C_j$ inside
$R[C]$ ($j=1,2$), and let $Z$ be the closed subscheme of $C_1$
determined by the ideal $J=(I_1+I_2)/I_1$ of $R[C_1]=R[C]/I_1$
(i.e.~$Z$ is the scheme-theoretic intersection of $C_1$ and $C_2$
inside $C_1$). Write $Z=S\cup T$ with closed subschemes $S$ and $T$ of
$C$ such that $S$ is supported on real points and $T$ is supported on
non-real points. Let $I_S=\sI_{R[C_1]}(S)$ and $I_T=\sI_{R[C_1]}(T)$
be the vanishing ideals of $S$ and $T$ in $R[C_1]$, so that
$J=I_S\cap I_T$. Since $T$ is non-empty by hypothesis, we have
$(0)\subsetneq I_T\subsetneq R[C_1]$. Put $A=R[C_1]/I_T^2$, and choose
$h\in I_T\setminus I_T^2$.  (Note that $I_T\neq I_T^2$ by the Krull
intersection theorem; see Bourbaki \cite{MR0171800}, \S 3.2.) Since
$T$ has non-real support, we have $\sqrt{-1}\in A$, so every element
of $A$ is a sum of squares; in particular, $h+I_T^2$, the class of $h$
in $A$, is a sum of squares in $A$. Therefore, the element
$(0+I_S,h+I_T^2)\in (R[C_1]/I_S)\times A$ is a sum of squares in
$(R[C_1]/I_S)\times A\isom R[C_1]/(I_S\cap I_T^2)$. Thus there exists
a sum of squares $f\in\sum R[C_1]^2$ that restricts to
$(0+I_S,h+I_T^2)$ modulo $I_S\cap I_T^2$. Since $h\in I_T$, we have
$f\in J$.

Now choose $F\in I_2\subset R[C]$ such that $F|_{C_1}=f$. Clearly,
$F$ is psd on $C(R)$. But we claim that $F$ cannot be a sum
of squares in $R[C]$. For if it were, say $F=\sum_{i=1}^r F_i^2$,
$F_i\in R[C]$, it would follow that $F_i\in I_2$ for all
$i=1,\dots,r$, since $C_2$ is real. This would imply
$f=\sum_{i=1}^r f_i^2$ with $f_i=F_i|_{C_1}$. From $F_i\in I_2$ we
could then conclude $f_i\in J$ for all $i=1,\dots,r$, hence $f\in J^2\subset (I_S\cap I_T^2)$, contradicting
the choice of $h$.
\end{proof}

In classifying all curves for which psd$=$sos, we first consider the
virtually compact case, i.e.~the case when every irreducible component
admits a non-constant bounded function. By Lemma
\ref{Lemma:VirtualCompactness}, this is equivalent to saying that all
irreducible components of $C$ have a non-real point at infinity. For
$R=\R$, Scheiderer has proved the following:

\begin{Thm}[Scheiderer \cite{MR2020709}, Cor.~4.15]\label{Thm:SOSredcurvVC}
Let $C$ be an affine curve over $\R$ with $C(\R)$ virtually
compact. If $C$ has no other real singularities than ordinary multiple
points with independent tangents, then every psd function in $\R[C]$
with only finitely many zeros is a sum of squares in $\R[C]$.
\end{Thm}

\begin{Cor}\label{Cor:SOSirrcurvVC}
Let $C$ be an irreducible affine curve over $\R$ with $C(\R)$ virtually compact. If $C$ has no other real singularities than ordinary multiple points with independent tangents, then psd$=$sos in $\R[C]$.\qed
\end{Cor}

\noindent For reducible curves, we have the following

\begin{Thm}\label{Cor:sosredcurvvc}
Let $C$ be a real affine curve over $\R$ with $C(\R)$ virtually
compact. Then psd$=$sos in $\R[C]$ if and only if the following
conditions are satisfied:
\begin{enumerate}
\item All real points of $C$ are ordinary multiple points with
independent tangents.
\item All intersection points of $C$ are real.
\end{enumerate}
\end{Thm}

\begin{proof}
Sufficiency of (1) and (2) follows directly from Scheiderer's theorem: Let $C_1,\dots,C_m$ be the irreducible components of $C$. By
Cor.~\ref{Cor:FnOnRedCurves}, we have
\[
\R[C]\isom\biggl\{(f_i)\in\prod_{i=1}^m\R[C_i]\:\bigl|\: f_i(P)=f_j(P)\text{ for
all }P\in C_i\cap C_j,1\le i,j\le m\biggr\}.
\]
Let $f=(f_1,\dots,f_m)\in\R[C]$ be psd. Upon relabelling, we may
assume that $f$ has only finitely many zeros on $C_1,\dots,C_l$ and
vanishes identically on $C_{l+1},\dots,C_m$ for some $0\le l\le
m$. Put $C'=\bigcup_{i=1}^l C_i$. Then $g=f\vert_{C'}$ is a sum of
squares in $\R[C']$ by the preceding theorem, say $g=\sum g_i^2$,
$g_i\in\R[C']$. For every $P\in C'\cap C_j$, $j > l$, we have
$g(P)=f(P)=0$, and since $P$ is real by hypothesis, it follows that
$g_i(P)=0$ for all $i$. This implies that
$f_i=\bigl(g_i\vert_{C_1},\dots,g_i\vert_{C_l},0,\dots,0\bigr)$ is a
function on $C$, by the above description of $\R[C]$, and that $f=\sum
f_i^2$.

Conversely, condition (1) is necessary by Prop.~\ref{Prop:BadRealSing}, condition (2) by Prop.~\ref{Prop:redcurvenonrealintersection}. 
\end{proof}

\begin{Example}
Let $C=\{(x^2+y^2-1)((x-1)^2+y^2-1)=0\}$,
two intersecting circles in the plane. Then psd$=$sos in $R[C]$.
But psd$\neq$sos in $R[C]$ for $C=\{(x^2+y^2-1)((x-3)^2+y^2-1)=0\}$, since the two circles intersect at a non-real point.
\end{Example}

We now turn to the case furthest from virtual compactness, namely that of a real affine curve $C$ that does not admit any bounded functions ($B(C)=R$). Again, the complete answer is known for the irreducible case, even for an arbitrary real closed ground field:

\begin{Thm}[Scheiderer]\label{Thm:SOSirrCurvesNVC}
Let $C$ be an irreducible affine curve over $R$. Assume that
$B(C)=R$. Then psd$=$sos in $R[C]$ if and only if $C$ is isomorphic to
an open
subcurve of $\A_R^1$. 
\end{Thm}

\begin{proof}
It has already been noted in~\ref{Def:RBF} that $B(C)=R$ holds if and
only if all points of $C$ at infinity are real. This is the way the
hypothesis is stated in Scheiderer \cite{MR2020709}, Thm.~4.17. The
proof for psd$=$sos for open subcurves of $\A^1_R$ easily reduces to
the case of the polynomial ring in one variable; see also Scheiderer
\cite{MR1675230}, Prop.~2.17.
\end{proof}

Note that an open subcurve of $\A^1_R$ is just the complement of
finitely many points. Under the hypothesis $B(C)=R$, all those points must
be real. These are exactly the non-singular, irreducible, rational,
affine curves over $R$ with $B(C)=R$.

\medskip

For reducible curves whose components are non-singular and rational, the condition for psd$=$sos will depend on the configuration of the irreducible components, so we need some combinatorial preparations: We associate with a curve $C$ (over any field $k$) a finite graph $\Gamma_C$\label{sym:gammac} as follows: The vertices of $\Gamma_C$ are the irreducible components of $C$, and we put an edge between two distinct vertices for every intersection point of the corresponding components. (The definition of $\Gamma_C$ has been changed compared to \cite{MeineDiss} after a suggestion by Michel Coste, simplifying the arguments that follow.) Recall that a \emph{simple cycle} of a graph is a subgraph that is homeomorphic to $S^1$. A graph that does not contain any simple cycles is called a \emph{forest} (a \emph{tree} if it is also connected).

\begin{Lemma}\label{Lemma:RelabelCurves}
Let $C$ be a curve $k$ with irreducible components
$C_1,\dots,C_m$. Then the graph $\Gamma_C$ is a forest if and only if
$C_1,\dots, C_m$ can be relabelled in such a way that
$C_i\cap(C_1\cup\cdots\cup C_{i-1})$ consists of at most one point for
every $1 < i\le m$.
\end{Lemma}

\begin{proof}
Clearly, we may assume that $C$ is connected, otherwise we can treat
all connected components separately. We prove the result by induction on $m$. The case $m=1$ is trivial, so assume that $m\ge 2$. It is elementary that a finite connected graph with $m$ vertices is a tree if and only if it has exactly $m-1$ edges. It follows that if $\Gamma_C$ is a tree, there exists a vertex of degree one. We may relabel and assume that this vertex corresponds to the component $C_m$. Put $C'=C_1\cup\cdots\cup C_{m-1}$. Then $C_m\cap C'$ consists of exactly one point. Furthermore, $\Gamma_{C'}$ is again a tree, so we are done by the induction hypothesis. Conversely, assume that $C_1,\dots,C_m$ are arranged such that $C_i\cap (C_1\cup\cdots\cup C_{i-1})$ consists of exactly one point for every $1<i\le m$. Again, write $C'=C_1\cup\cdots\cup C_{m-1}$. Then $\Gamma_{C'}$ is a tree by the induction hypothesis, hence it has $m-1$ vertices and $m-2$ edges. Since $C_m\cap C'$ consists of exactly one point, it follows that $\Gamma_C$ has $m$ vertices and $m-1$ edges, so it is a tree.
\end{proof}

If $C$ is a real curve with only real intersection points, the
condition on $\Gamma_C$ can sometimes be expressed in terms of the semialgebraic topology of $C(R)$:

\begin{Prop}\label{Prop:SimplyConnectedAndMixedCycles}
Let $R$ be a real closed field, and let $C$ be a curve over $R$ with
irreducible components $C_1,\dots,C_m$. Assume that all intersection
points of $C$ are real and that $C_i(R)$ is simply connected for all
$1\le i\le m$. Then $\Gamma_C$ is a forest if and only if
every connected component of $C(R)$ is simply-connected.
\end{Prop}

\begin{proof}
It suffices to note that, under the hypotheses, cycles in
$\Gamma_C$ correspond exactly to non-trivial $1$-cycles of $C(R)$.
\end{proof}

\begin{Example}
Even for real curves, it does not suffice to take only the real
picture into account if the $C_i(R)$ are not connected. For example,
let $C$ be the plane curve $\{(xy-1)(x-y)=0\}$, a hyperbola
intersecting a line. Clearly, $C(R)$ is simply connected, yet
$\Gamma_C$ is a simple cycle consisting of two vertices joint by two
edges.
\end{Example}

\begin{Thm}\label{Thm:sosredcurvnonvc}
Let $C$ be an affine curve over $R$ with irreducible components
$C_1,\dots,C_m$, and assume that $B(C_i)=R$ for all
$i=1,\dots,m$. Then psd$=$sos in $R[C]$ if and only if the following
conditions are satisfied:
\begin{enumerate}
\item All $C_i$ are isomorphic to open
subcurves of $\A^1_R$.
\item All intersection points of $C$ are real ordinary multiple points
with independent tangents.
\item The graph $\Gamma_C$ is a forest.
\end{enumerate}
\end{Thm}

\begin{proof}
Assume that $C$ satisfies all the conditions listed in the
theorem. Condition (1) implies that psd$=$sos in $R[C_1],\dots,R[C_m]$
by Thm.~\ref{Thm:SOSirrCurvesNVC}. Condition (3) implies by Lemma
\ref{Lemma:RelabelCurves} that we may rearrange the $C_i$ in such a
way that $C_i\cap (C_1\cup\dots\cup C_{i-1})$ consists of at most one
point for all $1\le i\le m$. Put $E_i=C_1\cup\cdots\cup C_i$ for all
$1\le i\le m$ and use induction on $i$: We already know that psd$=$sos in $R[E_1]$. For $i\ge 2$, psd$=$sos in $R[C_i]$ and in
$R[E_{i-1}]$ by the induction hypothesis. Now $C_i$ and $E_{i-1}$
have at most one intersection point which must then be a real ordinary
multiple point with independent tangents, by condition (2). Therefore,
\[
R[E_i]=\left\{\begin{array}{ll} R[C_i]\times R[E_{i-1}], & C_i\cap E_i=\emptyset\\
\bigl\{(f,g)\in R[C_i]\times R[E_{i-1}]\:\bigl|\: f(P)=g(P)\bigr\}, & C_i\cap
E_i=\{P\}\end{array}\right.
\]
by Cor.~\ref{Cor:FnOnRedCurves}. Thus psd$=$sos in $R[E_i]$
by Prop.~\ref{Prop:redcurvetransversal}. Eventually, we reach $i=m$,
and we see that psd$=$sos in $R[E_m]=R[C]$.

\medskip For the converse, assume that psd$=$sos in $R[C]$. By
Prop.~\ref{Prop:psdsos_passestoirrcomp}, psd$=$sos in $R[C_1], \dots,
R[C_m]$, so Thm.~\ref{Thm:SOSirrCurvesNVC} implies condition
(1). Furthermore, Prop.~\ref{Prop:BadRealSing} and
Prop.~\ref{Prop:redcurvenonrealintersection} imply condition (2). Thus
we are left with the case when conditions (1) and (2) are satisfied,
but (3) is not, i.e.~the graph $\Gamma_C$ contains a simple cycle. Let
$C_i$ be an irreducible component of $C$ corresponding to a vertex in
a simple cycle of $\Gamma_C$.  Let $C_i'=\bigcup_{j\neq i}C_j$ as
before. Then there is a connected component $E$ of $C_i'$ such that
$E\cap C_i$ contains at least two distinct points. It suffices to show
that psd$\neq$sos in $R[C_i\cup E]$ by Lemma
\ref{Prop:psdsos_passestoirrcomp}. So we may replace $C$ by $C_i\cup
E$ and assume right away that
$C$ and $C_i'$ are connected. Write $C_i\cap C_i'=\{P_1,\dots,P_r\}$, $r\ge 2$,
and let
\[
A=\bigl\{f\in R[C_i]\:\bigl|\: f(P_1)=\cdots=f(P_r)\bigr\}.
\]
From Cor.~\ref{Cor:FnOnRedCurves} and the fact that $C_i'$ is
connected, we see that the restriction map $R[C]\into R[C_i]$ induces
an isomorphism
\[
\bigl\{f\in R[C]\:\bigl|\: f\vert_{C_i'}\text{ is constant}\bigr\}\iso A.
\]
Now if an element of $A$ is a sum of squares in $R[C]$, then it is a
sum of squares in $A$ by Lemma \ref{Lemma:VirtualCompactness}
(7). (Note that since $C_i'$ is connected, $B(C_i)=R$ for all
$i=1,\dots,r$ implies $B(C_i')=R$, by Lemma
\ref{Lemma:VirtualCompactness} (5).)

We will make a similar argument as in the proof of
Prop.~\ref{Prop:redcurvenonrealintersection} and construct an element
of $A$ that is not a sum of squares in $A$ as follows: Fix an embedding $C_i\injects\P^1_R$, and let $\P^1 _R\setminus C_i=\{Q_1,\dots,Q_s\}$
be the points at infinity of $C_i$. Since $B(C_i)=R$, all $Q_j$ are
real by Lemma \ref{Lemma:VirtualCompactness} (3). Furthermore, since
$\P^1(R)$ is topologically a circle, we can relabel $P_1,\dots,P_r$
and assume that $P_j$ is next to $P_{j+1}$, i.e.~if $U_1$ and $U_2$
are the two connected components of $\P^1(R)\setminus\{P_j,P_{j+1}\}$,
either $U_1$ or $U_2$ contains none of $P_1,\dots,P_r$, for $1\le j\le
r-1$. We denote that connected component by $(P_j,P_{j+1})$. Assume
further that $Q_1\in (P_r,P_1)$.

Fix $j\in\{1,\dots,r\}$. Since $C_i$ is rational, there exists $h_j\in
R[C_i]$ such that
\[ {\rm div}_{\P^1_R}(h_j)=P_1+\cdots +\wh P_j+\cdots +P_r - (r-1)Q_1.
\]
Then $h_j(P_j)\neq 0$, and after multiplying with $(-1)^j\cdot
h(P_j)^{-1}$ we can assume that $h_j(P_j)=(-1)^j$. Put $f=\sum_j h_j$;
then $f(P_j)=(-1)^j$.  We have ${\rm ord}_{Q_1}(f)\ge -(r-1)$ and ${\rm ord}_P(f)\ge
0$ for all $P\in\P^1_R$, $P\neq Q_1$, so $f$ has poles only at $Q_1$,
and $f$ can have at most $r-1$ distinct zeros. But $f$ changes sign on
$(P_j,P_{j+1})$, so it must have a zero $R_j\in (P_j,P_{j+1})$ for
every $j=1,\dots,r-1$. We conclude that
\[ {\rm div}_{\P^1_R}(f)=R_1+\dots+R_{r-1}-(r-1)Q_1.
\]
Now since $f^2(P_j)=1$ for all $1\le j\le r$, we have $f^2\in A$, but
$f\notin A$, since $f(P_j)=(-1)^j$ for $j=1,\dots,r$ and $r\ge 2$. We
will show that $f^2$ cannot be a sum of squares in $A$: For if
$f^2=\sum f_j^2$ with $f_j\in R[C_i]$, then $f_j(R_l)=0$ for every
$1\le l\le r-1$ and every $j$, so each $f_j$ has at least $r-1$
distinct zeros. On the other hand, ${\rm ord}_{Q_l}(f_j)\ge {\rm ord}_{Q_l}(f)$ for
$l=1,\dots,s$ by Lemma~\ref{Cor:ValuationSOS}, so ${\rm ord}_{Q_1}(f_j)\ge
-(r-1)$ and ${\rm ord}_P(f_j)\ge 0$ for all $P\in\P^1_R$, $P\neq Q_1$. It
follows that $\div_{\P_R^1}(f_j)=R_1+\cdots+ R_{r-1} - (r-1)Q_1=\div_{\P_R^1}(f)$,
so there exists $c_j\in R$ such that $f_j=c_j f$, hence $f_j\notin A$.
\end{proof}

\begin{Cor}
Let $C$ be an affine curve over $R$ all of whose irreducible
components are isomorphic to $\A^1_R$. Then psd$=$sos in $R[C]$
if any only if the following conditions are satisfied:
\begin{enumerate}
\item All intersection points of $C$ are real ordinary multiple points
with independent tangents.
\item All connected components of $C(R)$ are simply connected.
\end{enumerate}
\end{Cor}

\begin{proof}
Combine the theorem with
Prop.~\ref{Prop:SimplyConnectedAndMixedCycles}.
\end{proof}

Combining Thm.~\ref{Thm:sosredcurvnonvc} with the preceding results,
we can now treat the general case of a reducible curve $C$ over $\R$
with irreducible components $C_1,\dots,C_m$ where $B(C_i)=\R$ holds
for some components of $C$ while $B(C_i)\neq\R$ holds for others:

\begin{Thm}\label{MainResult}
Let $C$ be an affine curve over $\R$, and let $C'$ be the union of all irreducible components $C_i$ of $C$ such that $B(C_i)=\R$. Then psd$=$sos in $\R[C]$ if and only if the following conditions are satisfied:
\begin{enumerate}
\item All real singularities of $C$ are ordinary multiple points with independent tangents.
\item All intersection points of $C$ are real.
\item All irreducible components of $C'$ are isomorphic to open
subcurves of $\A^1_\R$.
\item The graph $\Gamma_{C'}$ is a forest.
\end{enumerate}
\end{Thm}

\begin{proof}
Let $C_r$ be the union of all irreducible components of $C$ that are
real, $C_{nr}$ the union of those that are non-real. Condition (1)
implies that $C_{nr}(\R)=\emptyset$ and condition (2) implies that
$C_r\cap C_{nr}=\emptyset$, so if (1)--(4) are satisfied, psd$=$sos in
$\R[C]$ is equivalent to psd$=$sos in $\R[C_r]$ by
Cor.~\ref{Prop:Nonrealcomponents}. Conversely, if psd$=$sos in
$\R[C]$, then $C_{nr}(\R)=\emptyset$ and $C_r\cap C_{nr}=\emptyset$
hold by \ref{Prop:Nonrealcomponents}, so that (1)--(4) are satisfied
for $C$ if and only if they are satisfied for $C_r$. We may therefore assume that $C=C_r$, i.e.~$C$ is real. 

The necessity of (1) is Prop.~\ref{Prop:BadRealSing}, that of (2) is \ref{Prop:redcurvenonrealintersection}, that of (3) and (4) follows from Thm.~\ref{Thm:sosredcurvnonvc}. Conversely, assume that (1)--(4) are satisfied: 
Let $C''$ be the union of all irreducible components $C_i$ of $C$ such that $B(C_i)\neq\R$. Then $C''(\R)$ is virtually compact, hence psd$=$sos in $\R[C'']$ by Cor.~\ref{Cor:sosredcurvvc}. Also, psd$=$sos in $\R[C']$ by Thm.~\ref{Thm:sosredcurvnonvc}. Thus psd$=$sos in $\R[C]$ by Thm.~\ref{Thm:SOSOneCompactComp}.
\end{proof}

\section{Preorderings on curves}\label{Sec:POCurves}

The case of general preorderings instead of just sums of squares is
substantially harder, already in the irreducible case. (See \ref{Def:PO} for
basic notations and definitions used in this section). As far as the
irreducible case is concerned, we content ourselves here with citing
the results of Scheiderer for non-singular curves and the results of
Kuhlmann and Marshall for subsets of the line:

\begin{Thm}[Scheiderer \cite{MR2020709}, Thm.~5.17 and
\cite{MR1675230}, Thm.~3.5]\label{Thm:ScheidererPOCurvesVC} Let $C$ be
a non-singular, irreducible affine curve over $\R$, $\sH$ a finite subset
of $\R[C]$, $T=\PO(\sH)$, and $S=\sS(T)$. Assume that $B_C(S)\neq\R$
holds. Then $T$ is saturated if and only if the following conditions
are satisfied:
\begin{enumerate}
\item For every boundary point $P$ of $S$ in $C(\R)$ there exists an
element $h\in\sH$ such that ${\rm ord}_P(h)=1$.
\item For every isolated point $P$ of $S$ there exist $h_1,h_2\in\sH$
such that ${\rm ord}_P(h_1)={\rm ord}_P(h_2)=1$ and $h_1h_2\le 0$ holds in a
neighbourhood of $P$ in $C(\R)$.
\end{enumerate}
\end{Thm}

The results in the singular case are more complicated to state, and we
refer the reader to Scheiderer \cite{MR2020709}. The complementary
case $B_C(S)=\R$ is covered by the following result:

\begin{Thm}[Scheiderer \cite{MR1675230}, Thm.~3.5]\label{Thm:ScheidererPOCurvesNVC}\rule{0.1em}{0pt} 
Let $C$ be a non-singular, irreducible affine curve over $\R$, and let $S$ be a
basic closed subset of $C(\R)$ such that $B_C(S)=\R$. Then the
preordering $\sP_C(S)$ is finitely generated if and only if $C$ is an
open subcurve of $\A^1_\R$.
\end{Thm}

\medskip
In the case of the affine line, it is possible to say precisely
what the generators of the saturated preordering $\sP_C(S)$ must look
like:

\begin{Thm}[Kuhlmann-Marshall \cite{MR1926876}, Thm 2.2]\label{Thm:KMSubsetsLine}
Let $\sH$ be a finite subset of $\R[t]$, $T=\PO(\sH)$, and $S=\sS(T)$.
Assume that $B(S)=\R$ (i.e.~$S$ is not compact). Then $T$ is saturated
if and only if the following hold:
\begin{enumerate}
\item If $a=\min(S)$ exists, then $\lambda(t-a)\in\sH$ for some $\lambda>0$ in $\R$.
\item If $a=\max(S)$ exists, then $\lambda(a-t)\in\sH$ for some $\lambda>0$ in $\R$.
\item If $a<b\in S$ are such that $(a,b)\cap S=\emptyset$, then $\lambda(t-a)(t-b)\in\sH$ for some $\lambda>0$ in $\R$. 
\end{enumerate}
\end{Thm}

Properties (1)--(3) of the theorem specify a minimal set of generators for $T$
(unique up to positive scalars) that Kuhlmann and Marshall call the \emph{natural generators}. 

\medskip
We now turn our attention to reducible curves, but will only treat the
simplest case and some examples. We need the following

\begin{Lemma}
Let $C$ be a connected affine curve over $R$ with irreducible
components $C_1,\dots,C_m$. Assume that all intersection points of $C$
are ordinary multiple points with independent tangents and that the
graph $\Gamma_C$ is a forest. Then for all
$i=1,\dots,m$ and all $f\in R[C]$, there exists a unique
function $f_{(i)}\in R[C]$ such that $f_{(i)}|_{C_i}=f|_{C_i}$ and
$f_{(i)}|_{C_j}$ is constant for all $j\neq i$.
\end{Lemma}

\begin{proof}
The claim is trivial for $m=1$, so assume $m\ge 2$ and put
$C'=C_1\cup\cdots\cup C_{m-1}$. By Lemma \ref{Lemma:RelabelCurves}, we
may relabel and assume that $C_m\cap C'$ consists of a single point
$P$. Now $R[C]=\bigl\{(f,f')\in R[C_m]\times R[C']\:\bigl|\:
f(P)=f'(P)\bigr\}$ by Cor.~\ref{Cor:FnOnRedCurves}, so for every
$(f,f')\in R[C]$ we can form $(f'(P),f')\in R[C]$ and $(f,f(P))\in
R[C]$. From this the claim follows easily by induction.
\end{proof}

\begin{Prop}\label{Prop:SatPORedCurv}
Let $C$ be a connected affine curve over $R$ with irreducible
components $C_1,\dots,C_m$, let $\sH\subset R[C]$ be a finite subset,
and let $S=\sS(\sH)$ and $T=\PO(\sH)$. Assume that the following
conditions are satisfied:
\begin{enumerate}
\item The induced preordering $T|_{C_i}$ is saturated for all
$i=1,\dots,m$.
\item All intersection points of $C$ are real ordinary multiple points
with independent tangents and are contained in $S$.
\item The graph $\Gamma_C$ is a tree. 
\item If $h\in\sH$, then $h_{(i)}\in\sH$ for all $i=1,\dots,m$, where $h_{(i)}$ is defined as in the preceding lemma.
\end{enumerate}
\end{Prop}
\noindent Then $T$ is saturated.

\begin{proof}
The claim is trivial for $m=1$ because of condition (1), so assume
$m\ge 2$. Write $C=C_m\cup C'$, $C_m\cap C'=\{P\}$, as in the proof of
the preceding lemma using hypotheses (2) and (3), and write
$\sH=\{H_1,\dots,H_r\}$, $H_i=(h_i,h_i')$. We first show that for every $f\in R[C_m]$
with $f\ge 0$ on $S\cap C_m(R)$, the function $F_{(m)}=(f,f(P))\in R[C]$ is
contained in $T$. By hypothesis (1), $f$ has a representation
$f=\sum_i s_i\ul h^i$ with $s_i\in\sum R[C_m]^2$. Thus we can write
$(f,f(P))=\sum_i (s_i,s_i(P)) (\ul h^i,\ul h^i(P))=\sum_i(s_i,s_i(P))\ul
H_{(m)}^i$, and $\ul H_{(m)}^i\in\sH$ by hypothesis (4); thus $(f,f(P))\in T$,
as claimed. In the same way, we show $(f'(P),f')\in T$ for every
$f'\in R[C']$ with $f'\ge 0$ on $S\cap C'(R)$, using the induction hypothesis
instead of (1).

Now let $F=(f,f')\in R[C]$, and assume that $F\ge 0$ holds on $S$. If
$F(P)=0$, then $F=(f,0)+(0,f')\in T$ by what we have just shown. If
$F(P)\neq 0$, then $F(P)>0$ by hypothesis (2), and we can write
$F=(f,f(P))(1,f(P)^{-1}f')\in T$.
\end{proof}

\begin{Remark*}
Note that we also could have proved Prop.~\ref{Prop:redcurvetransversal}
by the same method employed here.
\end{Remark*}

\begin{Examples}\label{Ex:ReduciblePO}
\begin{enumerate}
\item Let $C$ be the plane curve $\{xy=0\}$ in $\A^2_R$. Write
$R[C]=\bigl\{(f,g)\in R[u]\times R[v]\:\bigl|\: f(0)=g(0)\bigr\}$, and
consider the preordering $T=\PO((u,0),(0,v))$ defining the
semialgebraic set $S=\sS(T)=\{(x,y)\in C(R)\:|\: x\ge 0\land y\ge
0\}$.  By the proposition, $T$ is saturated. However, the proposition
does not apply to the preodering $\PO((u,v))$ that defines the same
set, because condition (4) is violated. In fact,
$(u,0)\notin\PO((u,v))$ since $(u,0)=(s_1,t_1)+(s_2,t_2)(u,v)$ with
$s_i\in\sum R[u]^2$ and $t_i\in\sum R[v]^2$ would imply
$t_1=t_2=s_1=0$ by degree considerations and $s_2=1$, but $(1,0)\notin
R[C]$. Thus $\PO((u,v))$ is not saturated.

\item Let $C$ be as before, and let $T=\PO((u^2-1,v^2-1))$, so that
$S=\sS(T)=\{(x,y)\in C(R)\:|\: |x|\ge 1\land |y|\ge 1\}$. Here, the
intersection point $(0,0)$ is not contained in $S$, and the
proposition does not apply.  In fact, $T$ is not saturated: Let $f\in
R[u]$ be any quadratic polynomial that is non-negative on
$S\cup\{(0,0)\}$ but not psd (i.e.~not a sum of squares in
$R[u]$). For example, $f=(u-\frac 14)(u-\frac 34)$ will do.  Then
$(f,f(0))$ is a function in $R[C]$ that is non-negative on $S$ but not
contained in $T$: For $(f,f(0))=(s_1,t_1)+(s_2,t_2)(u^2-1,v^2-1)$
would imply $t_1=f(0)$ and $t_2=0$ by degree considerations, so we
must have $s_2(0)=0$. On the other hand, $s_2\in R$ again because of
the degree of $f$, so $(f,f(0))=(s_1,t_1)$, a contradition. It is not
clear whether the preordering $\sP_C(S)$ is finitely generated. The
fact that the intersection point is not contained in $S$ makes
condition (4) of the proposition unfulfillable, and it is not clear
what a suitable replacement should look like.  The best guess for a
set of generators in this particular instance seems to be
$\{(u^2-1,v^2-1),(u(u-1),0),(u(u+1),0), (0,v(v-1)),(0,v(v+1))\}$, but I
have been unable to verify this.
\end{enumerate}
\end{Examples}

\section{Applications to the moment problem}\label{Sec:MP}

Given a closed subset $K\subset\R^n$, the existence part of the
$K$-moment problem of functional analysis asks for a characterisation
of those linear functionals $L\colon\R[x_1,\dots,x_n]\into\R$ for
which there exists a (positive) Borel-measure $\mu$ supported on $K$
such that $L(f)=\int_K f d\mu$ holds for all
$f\in\R[x_1,\dots,x_n]$. Let $V$ be a real affine algebraic variety,
and let $K=\sS(h_1,\dots,h_r)$ be a basic closed semialgebraic subset
of $V(\R)$, defined by elements $h_1,\dots,h_r\in\R[V]$; let
$T=\PO(h_1,\dots,h_r)$ be the corresponding finitely generated
preordering of $\R[V]$. In an algebraic reformulation of the
$K$-moment problem (see for example Schm\"udgen \cite{MR1979186}, and
Powers-Scheiderer \cite{MR1823953}), one says that $T$ has the
\emph{strong moment property} (SMP) if $L|_T\ge 0$ implies
$L|_{\sP(K)}\ge 0$ for every linear functional
$L\colon\R[V]\into\R$. In other words, $T$ has the strong moment
property if the cone $T$ and the psd cone of $K$ in $\R[V]$ cannot be
separated by any linear functional. In particular, any saturated
preordering has property (SMP). A classical result of Haviland says
that $L|_{\sP(S)}\ge 0$ is necessary and sufficient for the existence
of a measure representing $L$. So if $T$ has the strong moment
property, this condition is replaced by the more manageable condition
$L|_T\ge 0$. Apart from the functional-analytic importance, the strong
moment property can be seen simply as an approximation property for
psd elements by elements of $T$. It is a consequence of Schm\"udgen's
Positivstellensatz that $T$ has the strong moment property whenever
$K$ is compact; see \cite{MR1092173}. In 2004, Schm\"udgen went on to
prove the following:

\begin{Thm}[Schm\"udgen's fibration theorem \cite{MR1979186}]
\label{Thm:SchmuedgenFibreThm}
Let $V$ be an affine $\R$-variety, $T$ a finitely generated
preordering of $\R[V]$, and $K=\sS(T)$. Assume that $\phi\colon
V\into\A^m$ is a real polynomial map such that the closure of
$\phi(K)$ in $\R^m$ is compact.  Then $T$ has property (SMP) if and
only if $T|_{\phi^{-1}(a)}$ has property (SMP) for all $a\in\R^m$.
\end{Thm}

Schm\"udgen's proof relies heavily on operator theory. In the meantime,
Marshall and Netzer have (independently) developed more
elementary proofs; see Netzer \cite{MR2358493} and Marshall
\cite{MR2383959}, Ch.~4.

Note that $T|_{\phi^{-1}(a)}$ denotes the restriction of $T$ to the
reduced fibre $\phi^{-1}(a)$, i.e.~if we write $a=(a_1,\dots,a_m)$,
$\phi= (\phi_1,\dots,\phi_m)$ with $\phi_i\in B_V(K)$, and
$I_a=(\phi_i-a_i\;|\;i\in\{1,\dots,m\})$, then
$\phi^{-1}(a)=\sV(\sqrt{I_a})$ so that
$T|_{\phi^{-1}(a)}=(T+\sqrt{I_a})/\sqrt{I_a}$. In other words,
$T|_{\phi^{-1}(a)}$ is the preordering induced by $T$ in the
coordinate ring $\R[\phi^{-1}(a)]=\R[V]/\sqrt{I_a}$ of the fibre. Schm\"udgen states his
theorem for $T+I_a$, and some additional effort is required to
pass from his version to the one above; see Scheiderer
\cite{MR2182447}, section 4, or \cite{MeineDiss}, Lemma 2.3.

\medskip
We can apply Schm\"udgen's fibration theorem to reduce the
moment problem for curves to the case where there do not exist any
non-constant bounded functions. In that case, the moment property is
equivalent to saturatedness, by a result of Powers and Scheiderer. 

\begin{Prop}\label{Thm:ReductionFibreThmCurves}
Let $C$ be an affine curve over $\R$, let $T$ be a finitely generated
preordering of $\R[C]$, and put $K=\sS(T)$. Let $C'$ be the union of all
irreducible components $C_i$ of $C$ for which $B_{C_i}(K\cap
C_i(\R))=\R$. The following are equivalent:
\begin{enumerate}
\item $T$ has the strong moment property;
\item $T|_{C'}$ has the strong moment property;
\item $T|_{C'}$ is saturated.
\end{enumerate}
\end{Prop}

\begin{proof}
(1) implies (2) since the moment property is preserved when passing to
a closed subvariety; see \cite{MR2182447}, Prop.~4.6. By Thm.~2.14
in Powers and Scheiderer \cite{MR1823953}, $T|_{C'}$ is closed in
$\R[C']$ (with respect to the finest locally convex
topology). Therefore, saturatedness and the strong moment property are
equivalent for $T|_{C'}$. Finally, (2) implies (1). For if $C_i$ is
any component of $C$ such that $B_{C_i}(K\cap C_i(\R))\neq\R$, put
$C''=C'\cup C_i$. By Lemma \ref{Lemma:VirtualCompactness} (5), we may
choose $f\in\R[C'']$ with $f|_{C_i}\in B_{C_i}(K\cap C_i(\R))\setminus\R$
and $f|_{C'}=0$. The fibres of $f\colon C''\into \A^1_\R$ are the
points of $C_i$ and the curve $C'$, so Schm\"udgen's fibre theorem
implies that $T|_{C''}$ has the strong moment property. Now continue
inductively until $C''=C$.
\end{proof}

Combined with Thm.~\ref{Thm:sosredcurvnonvc}, this proposition gives a complete set of necessary and sufficient conditions for the preordering $\sum\R[C]^2$ of sums of squares on an affine curve $C$ over $\R$ to have the strong moment property. Sufficient conditions for general preorderings on curves can be derived from Prop.~\ref{Prop:SatPORedCurv}.

\medskip
If $V$ is an irreducible affine surface and $K\subset V(\R)$ is such that there exists a non-constant bounded polynomial map $\phi\colon K\into\R^m$, then the fibres of $\phi$ are curves and Schm\"udgen's fibre theorem can be nicely combined with results for curves to say something about the moment problem for $K$. We want to stress here that reducible curves come up most naturally in this context: Even if $V$ and $\phi$ have good properties, one cannot expect all fibres of $\phi$ to be irreducible. Two examples:

\begin{Examples}
\begin{enumerate}
\item Let $T$ be the preordering of $\R[x,y]$ generated by $1-x^2y^2$, and put $K=\sS(T)$. The map $\phi\colon\R^2\into \R$ given by 
$\phi(x,y)=xy$ is obviously bounded on $K$. For $a\in\R$, put $C_a=\phi^{-1}(a)$. All fibres $C_a=\sV(xy-a)$ are hyperbolas for
$0\neq a\in [-1,1]$ and $T|_{C_a}=\sum\R[C_a]^2$. Since psd$=$sos in $\R[C_a]$ (Thm.~\ref{Thm:SOSirrCurvesNVC}), all $T_{C_a}$,
$0\neq a\in [-1,1]$, have property (SMP). The reducible fibre
$C_0=\sV(xy)$ is a pair of crossing lines and the induced preordering
$T|_{C_0}$ is again just $\sum\R[C_0]^2$ so that $T|_{C_0}$ has property (SMP),
too. Hence so does $T$ by Schm\"udgen's fibration theorem.

\item If we restrict to the first quadrant in the previous example,
i.e.~if we consider $T=\PO(x,y,1-xy)$, the same argument shows that
$T$ has property (SMP).  Here, we have to use the fact that $T|_{C_0}$
is the preordering generated by $x$ and $y$: Under the natural
isomorphism $\R[C_0]=\R[x,y]/(xy)\isom\R[t]\times\R[u]$, this means
that $T|_{C_0}$ is the preordering generated by $(t,0)$ and $(0,u)$
which is saturated (see Example \ref{Ex:ReduciblePO} (1)).

On the other hand, if $T=\PO(x+y,1-x^2y^2)$, then $T$ does not have
property (SMP), since the restriction $T|_{C_0}$ is the preordering
generated by $(t,u)$ which is not saturated by the discussion in
Example \ref{Ex:ReduciblePO} (1). It follows that $T|_{C_0}$ cannot
have property (SMP) either.
\end{enumerate}
\end{Examples}

{\linespread{1}
}
\end{document}